\documentclass[10pt]{article}
\usepackage{latexsym,amsmath,amsthm,amssymb}
\usepackage[sorted]{amsrefs}
\usepackage{graphicx}
\usepackage[colorlinks=true, linkcolor=blue]{hyperref}
\usepackage{geometry}
\usepackage{enumitem}
\usepackage{tikz}
\usepackage{mathtools}

\newtheorem{theorem}{Theorem}[section]
\newtheorem{proposition}[theorem]{Proposition}

\newtheorem{conjecture}[theorem]{Conjecture}

\newtheorem{question}[theorem]{Question}
\newtheorem{obs}[theorem]{Observation}

\setlist[enumerate]{topsep=0pt,partopsep=1ex,parsep=1ex}
\usepackage{dsfont}

\newcommand{\ZZ}{\ensuremath{\mathbb Z}}
\newcommand{\RR}{\ensuremath{\mathbb R}}

\newcommand{\floor}[1]{\left \lfloor #1 \right \rfloor}
\newcommand{\ceil}[1]{\left \lceil #1 \right \rceil}

\title{Covering triangular grids with multiplicity}

\author{Abdul Basit\thanks{Department of Mathematics, Xiamen University Malaysia; {\tt basit.abdul@gmail.com}.}
\and Alexander Clifton\thanks{Discrete Mathematics Group, Institute for Basic Science, Daejeon, South Korea {\tt yoa@ibs.re.kr}. The author is supported by the Institute for Basic Science (IBS-R029-C1) and in part by NSF Grant DMS-1945200.}
\and Paul Horn \thanks{Department of Mathematics, University of Denver, Denver, USA; {\tt paul.horn@du.edu}.  The author is supported in part by Simons Collaboration Grant \#525039.}}

\begin{document}

\maketitle
\begin{abstract}
Motivated by classical work of Alon and F\"uredi, we introduce and address the following problem: determine the minimum number of affine hyperplanes in $\mathbb{R}^d$ needed to cover every point of the \emph{triangular grid} $T_d(n) := \{(x_1,\dots,x_d)\in\mathbb{Z}_{\ge 0}^d\mid x_1+\dots+x_d\le n-1\}$ at least $k$ times. For $d = 2$, we solve the problem exactly for $k \leq 4$, and obtain a partial solution for $k > 4$. We also obtain an asymptotic formula (in $n$) for all $d \geq k - 2$. The proofs rely on combinatorial arguments and linear programming.
\end{abstract}
\section{Introduction}\label{sec:intro}

A classical theorem of Alon and F\"uredi \cite{AF}, answering a question of Komj\'ath \cite{K}, states that at least $n$ 
hyperplanes are necessary to cover the punctured hypercube $\{0,1\}^{n} \setminus \{\overrightarrow{0}\}$, if the origin must be uncovered.  This result, which can be proved by an application of the combinatorial Nullstellensatz~\cite{Alo99}, has been built on over the years.  Recent work of the second author and Huang \cite{CH} investigated the situation where each point on the hypercube (except the origin) must be covered at least $k$ times by the hyperplanes.  This was, in turn, built on by Sauermann and Wigderson \cite{SW} who considered more general polynomial coverings of the hypercube and by Bishnoi, Boyadzhiyska, Das, and den Bakker~\cite{BBBD} who considered hyperplane covers of two-dimensional rectangular grids.

In this paper, inspired by the previous work on covering the hypercube, we study covering another object: the triangular grid, defined as $T_d(n) := \{(x_1,\dots,x_d) \in \ZZ_{\geq{0}}^d\mid x_1+\dots+x_d\leq{n-1}\}$. Note that $T_d(n)$ is affinely equivalent to a triangular lattice in $\RR^d$ with $n$~points along each edge of the boundary.

While the triangular
grid could be considered a `punctured' version of a rectangular grid, this work (unlike the work on the hypercube and in \cite{BBBD}) does not try to avoid the punctured part.  The main challenge (and interest), then, in this setting comes from the fact that --- unlike in a hypercube or a rectangular grid --- the typical point 
of the grid does not lie in a hyperplane of maximum size. 

 Let $f(n,d)$ be the fewest number of affine hyperplanes needed to cover every point of $T_d(n)$ at least once. Noting that $f(n, d)$ is the optimum of an integer program, we consider the following linear programming relaxation: to every affine hyperplane $H$ in $\RR^d$, assign a nonnegative weight $w(H)$ such that 
$$ \sum_{H \ni p} w(H) \geq 1 \quad \text{ for every } p \in T_d(n). $$
We refer to such a weight assignment as a \emph{fractional cover}. Let $f^*(n,d)$ be the minimum size of a fractional cover, i.e. the minimum of $\sum_{H} w(H)$ over all fractional covers. Finally, we call a multiset of hyperplanes which covers every point of $T_d(n)$ at least $k$ times a $k$\emph{-cover} (a 1-cover is simply a \emph{cover}). Let $f(n,d,k)$ be the minimum cardinality of a $k$-cover of $T_d(n)$, so $f(n,d,k)$ denotes the fewest number of hyperplanes needed to cover every point of $T_d(n)$ at least $k$ times. Note that $f(n, d, 1) = f(n, d)$.

Here, we study the fractional and integral problems laid out: first in dimension $2$, then for higher dimensions.  In dimension 2, we completely solve the fractional problem (Theorems~\ref{1mod3}~and ~\ref{20mod3}), and then turn to the integral problem where we completely determine $f(n,2,k)$ for $k \leq 4$ (Theorem~\ref{2int}).  The most interesting thing we notice, however, in dimension $2$ is a seemingly strong connection between the solution to the fractional problem for the triangular grid of size $k$, $f^{*}(k,2)$ and the solution to the integral problem for covering arbitrarily large grids $k$ times, $f(n, 2,k)$.  In particular, it appears (Conjecture~\ref{duality})  that
\[ f(n, 2,k) = f^*(k,2)n + O_k(1).\footnote{We use $O_{v_1, v_2, \dots, v_k}$ to represent the usual big-$O$ notation where the constant of proportionality depends on the variables $v_1, v_2, \dots, v_k$. Unless otherwise specified, the big-$O$ notation represents asymptotics as $n \rightarrow \infty$.}   \]
While we are unable to completely resolve this question, we prove that $f^{*}(k,2)n +O_k(1)$ is an upper bound for $f(n,2,k)$ (Proposition \ref{upper}) and prove the conjecture under certain natural assumptions on the admissible lines in the covering (Theorem \ref{partialpr}).

In Section \ref{sec:highdim}, we turn to higher-dimensional analogues of the questions.  Perhaps the highlight here is Theorem \ref{3below}, which determines $f(n,d,k)$ asymptotically, so long as $d$ is sufficiently large compared with $k$.  We also show that the natural analogue of Conjecture \ref{duality}, mentioned above, fails for higher dimensions.  However, we conjecture an asymptotic formula for $f(n,d,k)$ for all $d$ and $k$ (Conjecture~\ref{megaconjecture}); this and other open problems are mentioned in Section \ref{sec:oq}. 

\paragraph{Notation and terminology}
By a \emph{standard hyperplane}, we mean a 
hyperplane of the form $x_i=c$ or $x_1+\dots+x_d=n-1-c$ with $c\in\{0,\dots,n-1\}$. In particular, standard hyperplanes contain at least one point of $T_d(n)$ and are parallel to a face of its convex hull. When $c = 0$, i.e., the hyperplane contains a face of the convex hull, we refer to the hyperplane as a \emph{bounding hyperplane}. A \emph{face} of $T_d(n)$ is the set $T_d(n) \cap H$ where $H$ is a bounding hyperplane.
If a hyperplane is not standard, we say it is \textit{non-standard}. Of course, for $d = 2$, we say lines instead of hyperplanes and, for $d = 3$, we say planes instead of hyperplanes.

\section{Triangular grids in two dimensions}\label{twodim}
Throughout this section, we will regard $T_2(n)$ as $\{(x,y)\in\mathbb{Z}_{\ge 0}^2\mid x+y\le n-1\}$.
\subsection{Fractional covering}\label{fractional}

\begin{theorem}\label{1mod3}
    $f^*(3j+1,2)=2j+1$ for all integers $j\geq{0}$.
\end{theorem} 

\begin{proof}
We begin by showing that $f^*(3j+1,2)\leq{2j+1}$. When $j=0$, $T_2(3j+1)$ consists of a single point, and we can cover it with one line. Suppose now that $j\geq{1}$, and note that $T_2(3j+1) = \{(x,y) \mid x,y\geq{0},x+y\leq{3j}\}$. Consider the fractional cover $\mathcal{C}$ where, for each $i \in \{0, \dots, 2j-1 \}$, the lines given by $x = i$, $y = i$ and $x+y = 3j-i$ have weight $(2j-i)/(3j)$, and  all other lines have weight 0.
    
To see that this is indeed a (fractional) cover, let $(a, b) \in T_2(3j+1)$. Note that at least one of $a$, $b$ is at most $2j-1$. Suppose, first, that both $a, b \leq 2j - 1$, then $(a,b)$ is contained in a vertical line with weight $(2j-a)/(3j)$ and a horizontal line with weight $(2j-b)/(3j)$, for a total weight of at least $(4j-a-b)/(3j)$. Hence, if $a+b \leq {j}$, then $(a, b)$ is covered by lines with total weight at least~1. On the other hand, if $a+ b > j$, then $(a,b)$ is contained in the line $x+y= a+b$ which has weight $(a + b - j)/(3j)$, for a total weight of $$\frac{4j-a-b}{3j} + \frac{a + b - j}{3j} = 1.$$

Suppose now, without loss of generality, that $a \leq{2j-1}$ and $b \geq{2j}$. Then, $(a,b)$ is contained in a vertical line with with weight $(2j-a)/(3j)$, and in a diagonal line with weight $(a+b-j)/(3j)$, for a total weight of $$\frac{2j-a}{3j} + \frac{a+b-j}{3j} = \frac{j+b}{3j} \geq{1}.$$
Hence, $\mathcal{C}$ is a fractional cover with total weight
\[ \frac{3}{3j} \sum_{i=0}^{2j-1}(2j-i) =2j+1. \]

We now establish the lower bound. For $j = 0$ this is straightforward. For $j\ge 1$, we assign nonnegative \emph{mass}, $m(p)$, to each point $p \in T_2(3j+1)$ such that the total mass assigned is $2j+1$. Then, showing that the set of points of $T_2(3j+1)$ contained in any line has total mass at most 1 implies the bound $f^*(3j+1, 2) \geq 2j+1$.

Our construction assigns a number of points a mass of 0. In particular, any point that is covered by lines of total weight more than 1 in $\mathcal{C}$ will be assigned a mass of 0.
Indeed, consider the quantity  $$ Q = \sum_{ (p, \ell)} w(\ell) m(p),$$
where the sum ranges over $(p, \ell)$ such that $p \in T_2(n)$, $\ell$ has positive weight (denoted by $w(\ell)$) in $\mathcal{C}$, and $p \in \ell$. Since every line contains points of total mass at most 1, we have $Q \leq 2j + 1$.
If there is a point with positive mass that is covered by lines of total weight more than $1$ in $\mathcal{C}$, then $$Q > \sum_{p\in T_2(3j+1)} m(p)= 2j+1,$$ a contradiction. On the other hand, since every point is contained in lines of total weight at least $1$, we have $Q \geq 2j+1$. This implies that implies that any line $\ell$ with $w(\ell) > 0$ contains points of mass exactly 1.

The points that receive positive mass can be partitioned into sets $X_1, \dots, X_j$ defined as follows. Fix $i \in \{1, \dots, j\}$, and note that the lines given by the equations ${x = j - i}$, ${x=j+i}$, ${y=j-i}$, ${y = j + i}$, ${x+y=2j-i}$, and $x+y=2j+i$ form a \emph{hexagon}. The set $X_i \subseteq T_2(3j+1)$ is defined to be the set of grid points contained on the boundary of this hexagon. See Figure~\ref{hexagonpic} for an illustration with $j = 3$. We refer to the sets $X_1, \dots X_j$ as \emph{hexagons}.

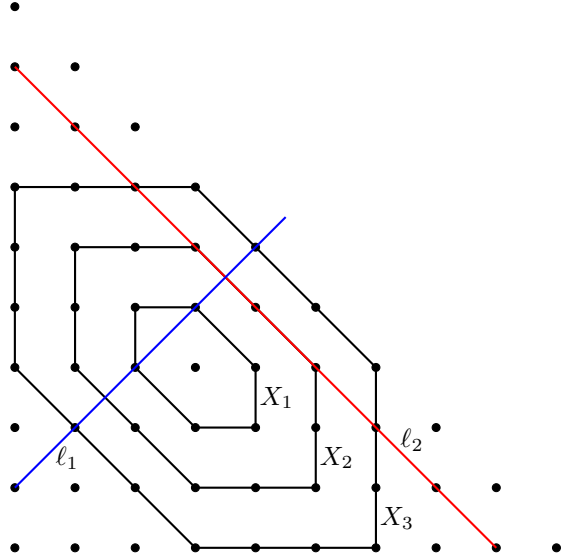
\begin{figure}[t]
    \centering
    \begin{tikzpicture}
[every node/.style={circle, draw=black!100, fill=black!100, inner sep=0pt, minimum size=3pt},
every edge/.style={draw, black!100, thick},
scale=0.8]

    \foreach \i in {0,...,9}{
        \pgfmathsetmacro{\z}{9-\i}
      \foreach \k in {0,...,\z}{
        \draw (\i,\k) node{};        
      };
      };

      \foreach \d in {1,2,3}{
      \pgfmathsetmacro{\m}{3-\d}
      \pgfmathsetmacro{\s}{3+\d}
      \draw[black!100,thick] (\m,\s) -- (\m,3) -- (3,\m) -- (\s,\m)-- (\s,3)--(3,\s)-- (\m,\s);
      };
      
      \node[draw=none,fill=none] at (6.35,0.5) {$X_3$};
      \node[draw=none,fill=none] at (5.35,1.5) {$X_2$};
      \node[draw=none,fill=none] at (4.35,2.5) {$X_1$};
      
      \draw[red!100,thick](0,8)--(8,0);
      \draw[blue!100,thick](0,1)--(4.5,5.5);
      
      \node[draw=none,fill=none] at (0.87,1.5) {$\ell_1$};
       \node[draw=none,fill=none] at (6.6,1.8) {$\ell_2$};

  \end{tikzpicture} \vspace{0.2in}\\

    \caption{Masses for $T_2(10)$ -- 
     Points in $X_1$ are assigned mass $1/12$, points in $X_2$ are assigned mass $1/6$, and points in $X_3$ are assigned mass $1/4$. All other points have mass $0$.}
    \label{hexagonpic}
\end{figure}

We impose the additional restriction that, for each hexagon, the points contained in are assigned equal mass. This implies that points of $X_j$ contained in lines given by $x = 0, y = 0,$ and $x + y = n - 1$ (and, hence, all points of $X_j$) receive a mass of $1/(j+1)$. We then successively assign mass to the points in $X_i$ for $i = j-1, \dots, 1$, while ensuring that the lines $x=j-i$, $y=j-i$, and $x+y=2j+i$ each contain points of total mass $1$. 
Following this strategy determines the following assignment of masses.

For each $i$, the points of $X_i$ are assigned a mass of $i/(j(j+1))$. All other points are assigned mass~$0$. See Figure~\ref{hexagonpic} for an illustration. Note that the hexagon $X_i$ has $i+1$ points along a side for a total of $6(i+1)-6 = 6i$ points. Thus, the total mass of all the points in $T_2(3j+1)$ is 
\[   \sum_{i=1}^j 6i\left(\frac{i}{j(j+1)}\right) = 2j+1.\] 

It remains to show that no line contains points of total mass more than $1$.  Suppose first that $\ell$ is a non-standard line, i.e., $\ell$ is not horizontal, vertical, nor has slope $-1$ (e.g., the line $\ell_1$ in Figure~\ref{hexagonpic}). Then $\ell$ can contain at most two points of each hexagon, implying that $\ell$ contains points with total mass at most
\[    2\sum_{i=1}^j \frac{i}{j(j+1)}=1. \]
Similarly, if $\ell$ is a standard line and does not contain a side of $X_i$ (e.g., if $(j+1, j+1) \in \ell$), then $\ell$ contains at most two points of each hexagon and so contains points with total mass at most 1. Finally, suppose $\ell$ contains a side of $X_s$ (e.g.,  the line $\ell_2$ in Figure~\ref{hexagonpic}). Then $\ell$ contains $s + 1$ points of $X_s$ and two points from $X_i$ for each $i > s$.  So $\ell$ contains points with total mass
\begin{align*}
(s+1)\left(\frac{s}{j(j+1)}\right)+2\sum_{i=s+1}^j\frac{i}{j(j+1)} = 2 \sum_{i=1}^j\frac{i}{j(j+1)} = 1.
\end{align*}
\end{proof}

\begin{theorem}\label{20mod3}
For all integers $j\geq{0}$,
\begin{enumerate}[label = (\alph*)]
    \item \label{itm:dim2-2mod3} $\displaystyle f^*(3j+2,2)=2j+1+\frac{2j+1}{3j+2}$,
    \item \label{itm:dim2-3mod3} $\displaystyle f^*(3j+3,2)=2j+2+\frac{j+1}{3j+4}$.
\end{enumerate}
\end{theorem}

\begin{proof}
We begin by establishing the upper bounds. Since the proof here is similar to that of the upper bound in Theorem~\ref{1mod3}, we only present the fractional cover and leave the verification as an exercise.

Let $n = 3j + 2$ with $j \geq 0$. The upper bound in~\ref{itm:dim2-2mod3} is obtained by considering the fractional cover~$\mathcal{C}$ where, for $i \in \{0,\dots, 2j \}$, the lines given by $x = i$, $y = i$, and $x+y = 3j + 1-i$ have weight $(2j + 1 - i)/(3j + 2)$, and all other lines have weight 0.

Suppose now that $n = 3j + 3$ with $j \geq 0$. The upper bound in~\ref{itm:dim2-3mod3} is obtained by considering the fractional cover where, for $i \in \{0, \dots, 2j+1 \}$, the lines given by $x = i$, $y = i$, and $x+y = 3j + 2 -i$ have weight $(2j + 2 - i)/(3j + 4)$, and all other lines have weight 0.

We now proceed to establish the lower bound. As in the proof of Theorem~\ref{1mod3}, we will assign nonnegative mass to each point of $T_2(n)$ such that no line contains points of total mass more than $1$. This implies that $f^*(n,2)$ is at least the sum of all masses. 
The proof here is more involved than that of Theorem~\ref{1mod3}. Hence, we give a detailed proof for $n = 3j+2$. For $n = 3j+3$, the proof is similar to that of $n = 3j+2$ and we only present the construction. 

\emph{Proof of lower bound in~\ref{itm:dim2-2mod3}:} Let $n = 3j + 2$. As in the proof of Theorem~\ref{1mod3}, any point of $T_2(n)$ covered by lines of total weight more than $1$ in the upper bound construction $\mathcal{C}$ has mass 0, and any line with positive weight in $\mathcal{C}$ contains points of total mass 1.

For $j=0$, a cover is easily obtained by assigning each point a mass of $1/2$. From here on, we assume $j\ge 1$. The points with positive mass can be partitioned into disjoint hexagons $X_1,\dots,X_j$, and a triangle~$Y$. The triangle $Y$ consists of the three points $(j,j), (j+1, j)$, and $(j+1, j)$. For $i \in \{1, \dots, j\}$, $X_i$ consists of points of $T_2(n)$ lying on the hexagon formed by the lines given by ${x=j-i}$, ${x=j+i+1}$, ${y=j-i}$, ${y=j+i+1}$,  ${x+y=2j-i}$, and ${x+y=2j+i+1}$. Note that the hexagons are not equilateral -- the hexagon $X_i$ has three \emph{short} sides containing $i+1$ points, and three \emph{long} sides containing $i+2$ points. We will refer to a point as a \emph{corner} point if it lies on both a long and short side, or an \emph{interior} point if it is not a corner point. See Figure~\ref{hexagonpic2} for an illustration with $j = 3$.

\begin{figure}
    \centering
    \begin{tikzpicture}
[every node/.style={circle, draw=black!100, fill=black!100, inner sep=0pt, minimum size=3pt}, every edge/.style={draw, black!100, thick}, scale=0.73]

    \foreach \i in {0,...,10}{
        \pgfmathsetmacro{\z}{10-\i}
      \foreach \k in {0,...,\z}{
        \draw (\i,\k) node{};
         check if (\i,\j) > (2,2) };};

      \draw[black!100, thick](3,3)--(4,3)--(3,4)--(3,3);
      
      \draw[blue!100,thick] (3,2) -- (5,2) -- (5,3) -- (3,5)-- (2,5)--(2,3)-- (3,2);
      \draw[red!100,thick](5,2)--(5,3);
       \draw[red!100,thick](3,5)--(2,5);
        \draw[red!100,thick](2,3)--(3,2);
      
      \draw[blue!100,thick] (3,1)--(6,1)--(6,3)--(3,6)--(1,6)--(1,3)--(3,1);
      \draw[red!100,thick](6,1)--(6,3);
       \draw[red!100,thick](3,6)--(1,6);
        \draw[red!100,thick](1,3)--(3,1);
      
      \draw[blue!100,thick] (3,0)--(7,0)--(7,3)--(3,7)--(0,7)--(0,3)--(3,0);
       \draw[red!100,thick](7,0)--(7,3);
       \draw[red!100,thick](3,7)--(0,7);
        \draw[red!100,thick](0,3)--(3,0);
 \node[draw=none,fill=none] at (7.4,0.5) {$X_3$};
      \node[draw=none,fill=none] at (6.4,1.5) {$X_2$};
      \node[draw=none,fill=none] at (5.4,2.5) {$X_1$};
      \node[draw=none,fill=none] at (4.3,3){$Y$};

  \end{tikzpicture} \vspace{0.2in}\\

    \caption{Masses for $T_2(11)$ --  corner points and interior points on the short sides (red) of $X_1, X_2,$ and $X_3$ have mass $5/44$, $2/11$, and $7/44$ respectively. Interior points on the long sides (blue) of $X_1, X_2,$ and $X_3$  have mass $1/11$, $7/44$, and $5/22$. Points on triangle $Y$ have mass $1/22$. All other points have mass $0$.}

\label{hexagonpic2}\end{figure}
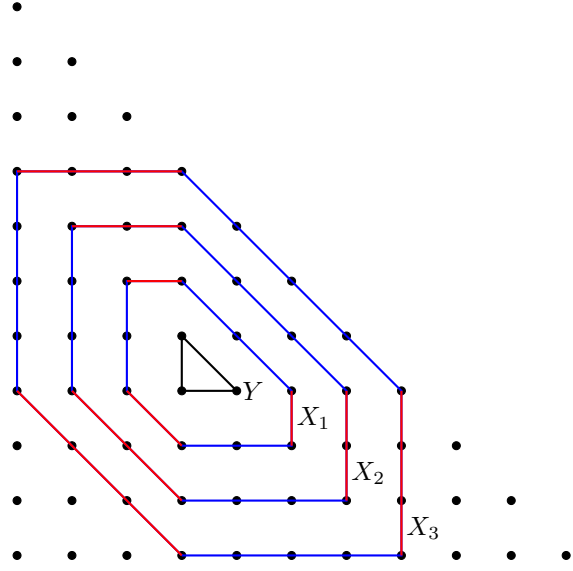

Recall that points not contained in hexagons or $Y$ have mass $0$ and that each line with positive weight in the fractional cover $\mathcal{C}$ contains points of total mass exactly $1$. With these requirements, we obtain that each standard line except for ones that contain a short side of $X_j$ contain points of total mass either $0$ or $1$. This forces the short sides of $X_j$ to receive the \emph{leftover mass}. Since $\mathcal{C}$ assigns positive weight to $2j+1$ standard lines in each direction, the short sides of $X_j$ each contain points of total mass $(2j+1)/(3j+2)$.

We first assign the corner and interior points on the short sides of $X_j$ equal mass so that the lines containing these sides contain points of total mass $(2j+1)/(3j+2)$. Then, we assign equal mass to the remaining points of $X_j$ (interior points on the long sides) so that the lines given by $x=0, y=0$, and $x+y=n-1$ each contain points of total mass $1$. 
We then successively assign mass to the points in $X_i$ for $i = j-1, \dots, 1$. For each $i$, we first assign equal mass to the points on the short sides of $X_i$ such that the lines containing the short sides of $X_i$ each contain points of total mass $1$. This then determines the total mass of the interior points on the long sides of $X_i$, which we distribute equally. Finally, once we have assigned mass to each hexagon, the requirement that each line with positive weight in the fractional cover $\mathcal{C}$ contains points of total mass 1 determines the masses of points in $Y$.

The strategy outlined above results in the following mass assignment: 
\begin{itemize}
    \item the corner points and interior points on the short sides of $X_j$ have mass $(2j+1)/((j+1)(3j+2))$ for a total mass of $$ 3(j+1)\left(\frac{2j+1}{(j+1)(3j+2)}\right) = \frac{3(2j+1)}{3j+2},$$
    \item for $1 \leq i \leq j-1$, the corner points and interior points on the short sides of $X_i$ have mass $(3i+2)/((j+1)(3j+2))$ for a total mass of
    $$ 3\sum_{i=1}^{j-1}(i+1)\left(\frac{3i+2}{(j+1)(3j+2)}\right) = \frac{3(j^3 + j^2 - 2)}{(j+1)(3j+2)},$$
    \item for $1 \leq i \leq j$, the interior points on the long sides of $X_i$ have mass $(3i+1)/((j+1)(3j+2))$ for a total mass of
    $$ 3\sum_{i=1}^{j}i\left(\frac{3i+1}{(j+1)(3j+2)}\right) = \frac{3j(j+1)}{(3j+2)},$$    
    \item points of $Y$ have mass $2/((j+1)(3j+2))$ for a total mass of
    $$\frac{6}{(j+1)(3j+2)}.$$    
\end{itemize}
The total mass of points in $T_2(3j+2)$ is
$$ \frac{3(2j+1)}{3j+2} + \frac{3(j^3 + j^2 - 2)}{(j+1)(3j+2)} + \frac{3j(j+1)}{3j+2} + \frac{6}{(j+1)(3j+2)} = 2j + 1 + \frac{2j+1}{3j+2}.$$

It remains to show that no line contains points of total mass more than $1$. Suppose first that $\ell$ is a standard line, i.e., $\ell$ is horizontal, vertical, or has slope $-1$. It suffices to consider the following cases:

\emph{Case} (i): $\ell$ contains the short side of a hexagon. If this hexagon is $X_j$, then $\ell$ contains exactly $j+1$ points of $T_2(n)$ for a total mass of
$$(j+1)\left(\frac{2j+1}{(j+1)(3j+2)}\right)<1.$$
Suppose now that $\ell$ contains the short side of $X_s$ with $s < j$. Then $\ell$ contains two interior points on the long sides of $X_i$ for each $i > s$. The total mass of points contained in $\ell$ is
$$ (s+1)\left(\frac{3s+2}{(j+1)(3j+2)}\right)+2\sum_{i=s+1}^j \frac{3i+1}{(j+1)(3j+2)} = 1.$$

\emph{Case} (ii): $\ell$ contains the long side of a hexagon. If this hexagon is $X_j$, then $\ell$ contains exactly $j+2$ points with positive mass, for a total mass of
$$ 2\left(\frac{2j+1}{(j+1)(3j+2)}\right)+j\left(\frac{3j+1}{(j+1)(3j+2)}\right) = 1.$$ 
Suppose now that $\ell$ contains the long side of $X_s$ with $s < j$. Then $\ell$ contains two interior points on the short sides of $X_i$ for each $i > s$. The total mass of points contained in $\ell$ is
$$    s\left(\frac{3s+1}{(j+1)(3j+2)}\right)+2\left(\sum_{i=s}^{j-1}\frac{3i+2}{(j+1)(3j+2)}\right)+2\left(\frac{2j+1}{(j+1)(3j+2)}\right) =1. $$
    
\emph{Case} (iii): $\ell$ contains a side of $Y$ and two corner points of $X_i$ for each $i$. The total mass of points contained in $\ell$ is:
    $$ 2\left(\frac{2}{(j+1)(3j+2)}\right)+2\left(\sum_{i=1}^{j-1}\frac{3i+2}{(j+1)(3j+2)}\right)+2\left(\frac{2j+1}{(j+1)(3j+2)}\right) = 1. $$
    
\emph{Case} (iv): $\ell$ contains one point of $Y$ and two interior points on the long sides of each hexagon. The total mass of points contained in $\ell$ is:
$$ \frac{2}{(j+1)(3j+2)}+2\left(\sum_{i=1}^j\frac{3i+1}{(j+1)(3j+2)}\right) =1.$$

We now assume that $\ell$ is a non-standard line. Suppose, for contradiction, that $\ell$ contains points of total mass more than~$1$. 
Let $p, q \in \ell$ be the points with largest mass. Note that the mass of each of $p, q$ is at most $M:=(3j+1)/((j+1)(3j+2))$, the largest mass assigned to any point. Furthermore, since the only points assigned a mass of $M$ lie on a long side of $X_j$, $\ell$ can contain at most two points with this mass. For $j>1$, every other mass is at most $(3j-1)/((j+1)(3j+2))$, so the number of points contained in $\ell$ with positive mass is at least 
$$   2 + \frac{1- 2\left((3j+1)/((j+1)(3j+2))\right)}{(3j-1)/((j+1)(3j+2))} \geq j+2. $$
For $j=1$, we have $2M<1$ and, hence, the number of points contained in $\ell$ with positive mass is at least $j+2$.

Suppose $(a, b), (a + m_1, b + m_2) \in T_2(n)$ with  $m_1, m_2$ non-zero integers are consecutive grid points on $\ell$. The number of rows of $T_2(n)$ with points of non-zero mass is $2j+2$. The same is true for columns of $T_2(n)$. A non-standard line can contain at most one point in any row or column. Since $\ell$ contains at least $j + 2$ points, this implies $|m_1|, |m_2| = 1$. Since $\ell$ is non-standard, it follows that $\ell$ has slope $1$.

Observe that the mass assignment is symmetric up to affine transformations of $\RR^2$ that map $T_2(n)$ to itself. Furthermore, there are affine transformations which map $\ell$ to a line $\ell'$ with slope $-2$ or $-1/2$. But, by the argument above, $\ell'$ contains points of total mass at most one. Thus, $\ell$ contains points of total mass at most one, a contradiction.

\emph{Proof of lower bound in~\ref{itm:dim2-3mod3}:} Let $n = 3j+3$. We assign mass to points of $T_2(n)$ for a total mass of $2j+2+(j+1)/(3j+4)$. The aim here is to be brief --- we simply describe the construction and leave verification as an exercise.

For $j=0$, a cover is easily obtained by assigning mass $1/4$ to the corners of $T_2(n)$ and mass $1/2$ to the remaining points. From here on, we assume $j\ge 1$. Any point that is covered by lines with total weight more than $1$ gets mass zero. The remaining points can be partitioned into disjoint hexagons $X_1,\dots,X_j$, and a triangle $Y$. The triangle $Y$ consists of the six points lying on the boundary of the triangle with vertices $(j,j)$, $(j,j+2)$, and $(j+2,j)$. The hexagon $X_i$ consists of the grid points of $T_2(n)$ on the hexagon formed by the lines $x = j-i, x=j+i+2, y=j-i, y=j+i+2, x+y=2j-i,$ and $x+y=2j+i+2$. Note that the hexagons are not equilateral -- the hexagon $X_i$ has three \emph{short} sides containing $i+1$ points, and three \emph{long} sides containing $i+3$ points. We will refer to a point as a \emph{corner} point if it lies on both a long and short side, or an \emph{interior} point if it is not a corner point. Corner and interior points of~$Y$ are defined in the obvious manner.
See Figure~\ref{hexagonpic3} for an illustration with $j = 3$.

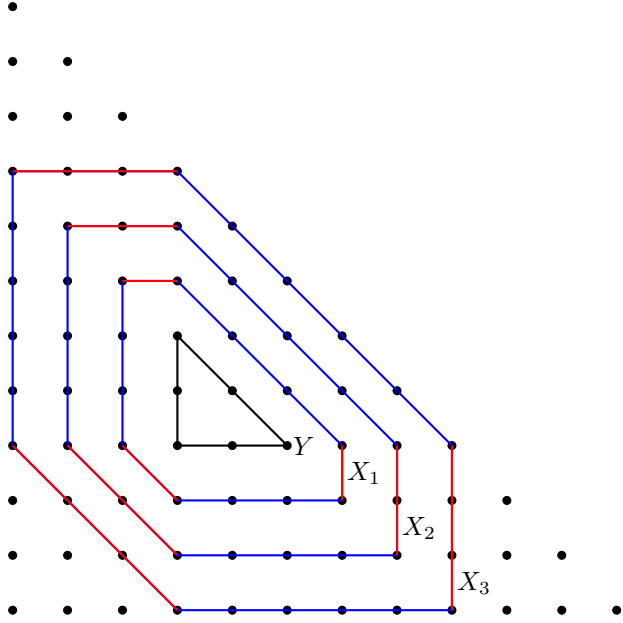
\begin{figure}
    \centering
    \begin{tikzpicture}
[every node/.style={circle, draw=black!100, fill=black!100, inner sep=0pt, minimum size=3pt}, every edge/.style={draw, black!100, thick}, scale=0.73]

    \foreach \i in {0,...,11}{
        \pgfmathsetmacro{\z}{11-\i}
      \foreach \k in {0,...,\z}{
        \draw (\i,\k) node{};
         check if (\i,\j) > (2,2) };};

      \draw[black!100, thick](3,3)--(5,3)--(3,5)--(3,3);
      
      \draw[blue!100,thick] (3,2) -- (6,2) -- (6,3) -- (3,6)-- (2,6)--(2,3)-- (3,2);
      \draw[red!100,thick](6,2)--(6,3);
       \draw[red!100,thick](3,6)--(2,6);
        \draw[red!100,thick](2,3)--(3,2);
      
      \draw[blue!100,thick] (3,1)--(7,1)--(7,3)--(3,7)--(1,7)--(1,3)--(3,1);
      \draw[red!100,thick](7,1)--(7,3);
       \draw[red!100,thick](3,7)--(1,7);
        \draw[red!100,thick](1,3)--(3,1);
      
      \draw[blue!100,thick] (3,0)--(8,0)--(8,3)--(3,8)--(0,8)--(0,3)--(3,0);
       \draw[red!100,thick](8,0)--(8,3);
       \draw[red!100,thick](3,8)--(0,8);
        \draw[red!100,thick](0,3)--(3,0);
 \node[draw=none,fill=none] at (8.4,0.5) {$X_3$};
      \node[draw=none,fill=none] at (7.4,1.5) {$X_2$};
      \node[draw=none,fill=none] at (6.4,2.5) {$X_1$};
      \node[draw=none,fill=none] at (5.3,3){$Y$};

  \end{tikzpicture} \vspace{0.2in}\\

    \caption{Masses for $T_2(12)$ --
    corner points and interior points on the short sides (red) of $X_1, X_2,$ and $X_3$ have mass $7/52$, $5/26$, and $1/13$ respectively. Interior points on the long sides (blue) of $X_1, X_2,$ and $X_3$  have mass $5/52$, $2/13$, and $11/52$. The corner points of $Y$ have mass $1/13$, and the interior points have mass $1/26$. All other points have mass $0$.}
    \label{hexagonpic3}\end{figure}

The mass is assigned as follows: 
\begin{itemize}
    \item the corner points and interior points on the short sides of $X_j$ have mass $1/(3j+4)$ for a total mass of
    $$ 3(j+1)\left(\frac{1}{3j+4}\right) = \frac{3j+3}{3j+4},$$
    \item for $1 \leq i \leq j-1$, the corner points and interior points on the short sides of $X_i$ have mass $(3i+4)/((j+1)(3j+4))$ for a total mass of
    $$ 3\sum_{i=1}^{j-1}(i+1)\left(\frac{3i+4}{(j+1)(3j+4)}\right) = \frac{3(j^3 + 2j^2 +j - 4)}{(j+1)(3j+4)},$$
    \item for $1 \leq i \leq j$, the interior points on the long sides of $X_i$ have mass $(3i+2)/((j+1)(3j+4))$ for a total mass of
    $$ 3\sum_{i=1}^{j}(i + 1)\left(\frac{3i+2}{(j+1)(3j+4)}\right) = \frac{3j(j^2 + 4j + 5)}{(j+1)(3j+4)},$$    
    \item the corner points of $Y$ have mass $4/((j+1)(3j+4))$ and the interior points have mass $2/((j+1)(3j+4))$ for a total mass of
    $$\frac{18}{(j+1)(3j+4)}.$$    
\end{itemize}
The total mass of points in $T_2(3j+3)$ is
$$ \frac{3j+3}{3j+4} + \frac{3(j^3 + 2j^2 +j - 4)}{(j+1)(3j+4)} + \frac{3j(j^2 + 4j + 5)}{(j+1)(3j+4)} + \frac{18}{(j+1)(3j+4)} = 2j + 2 + \frac{j+1}{3j+4}.$$
\end{proof}
The verification of the fact that each line contains points of total mass at most 1 follows similarly to the computation in the proof of Theorem~\ref{20mod3}~\ref{itm:dim2-2mod3}.

\subsection{Integer covering}\label{sec:intcovering}
The solution to the fractional problem gives an immediate lower bound on the integral problem:
\begin{equation}
\label{eq:fractionaltointegerbound}
f(n,2,k) \geq {kf^*(n,2)}.
\end{equation}
If $n$ is fixed and $k \rightarrow \infty$, then $f(n,2,k)=kf^*(n,2)+O_n(1)$ (see for example \cite{BBBD}, Proposition 1). On the other hand, for fixed $k$ as $n\to\infty$, we obtain via Theorems \ref{1mod3} and \ref{20mod3} that
$$f(n,2,k) \geq {kf^*(n,2)} = k\left(\frac{2n}{3}+O(1)\right)= \left(\frac{2k}{3}\right)n + O_k(1).$$

We show that this bound is not the best possible for $k\leq{4}$, and conjecture that this is the case for every value of $k$.

\begin{theorem}\label{2int}
    For all $n\geq{2}$,
\[f(n,2,k)= \begin{dcases}
    n &\mbox{if } k=1,\\
    \ceil{3n/2} &\mbox{if } k=2,\\
    \ceil{9n/4} &\mbox{if } k=3,\\
    3n &\mbox{if } k=4.
\end{dcases}
\]
\end{theorem}

\begin{obs}\label{boundary}
    If $H$ is a bounding hyperplane of $T_d(n)$, then $T_d(n) \setminus H$ is a copy of $T_d(n-1)$. Fix $d, k > 0$ and suppose there are constants $\alpha := \alpha_{d, k}$, $\beta := \beta_{d, k}$ such that $f(n-1,d,k) \geq \alpha(n-1)+ \beta$. Then any $k$-cover of $T_d(n)$ which contains a bounding hyperplane with multiplicity at least $\alpha$ has cardinality at least $\alpha n + \beta$. Thus, in proving a bound of the form $f(n,d,k) \geq \alpha n + \beta$ via induction on $n$, it suffices to consider $k$-covers in which each bounding hyperplane has multiplicity less than $\alpha$.
\end{obs}

\begin{proof}[Proof of Theorem \ref{2int}]

We begin by showing the lower bounds. For each $k$, the base case of $n = 2$ follows immediately from~\eqref{eq:fractionaltointegerbound}. Suppose now that $n > 2$. 

For $k \in \{1, 2, 4\}$, our proof relies on the following idea: if a line is not a bounding line, then it can cover at most two points contained in the bounding lines. Hence, any cover must either contain bounding lines with high multiplicity, or a large number of non-bounding lines. We present these proofs first. Unfortunately, this method fails to give an optimal bound for $k = 3$. Here, our argument is somewhat more involved and takes into account the multiplicity of other standard lines. In particular, we consider the smallest non-negative integer $c$ such that the line given by  $x_1 = c$ has low multiplicity in the cover. If $c$ is large, many lines parallel to $x_1 = 0$ have high multiplicity and the cover is large. On the other hand,  if $c$ is small, many lines are needed to cover the points of $x_1=c$ again implying that the cover is large.  We note that variations of both arguments are used repeatedly in Section~\ref{sec:highdim}.

\emph{The case $k = 1$:} Since we want to show $f(n, 2, 1) \geq n$, by Observation \ref{boundary}, it suffices to consider covers of $T_2(n)$ that do not contain any bounding lines. Let $H$ be the line given by $x = 0$. Any cover must contain a distinct line for each point in $T_2(n) \cap H$, implying that $f(n,2,1)\geq{n}$.

\emph{The case $k = 2$:} 
Recall that we want to show $f(n, 2, 2) \geq \ceil{3n/2}$. For $n=3$, the desired result again follows immediately from~\eqref{eq:fractionaltointegerbound}. By Observation \ref{boundary}, it suffices to consider 2-covers of $T_2(n)$ where the multiplicity of each bounding line is at most one.

When $n\geq{4}$, let $\mathcal{C}$ be a 2-cover of $T_2(n)$ where exactly $m$ bounding lines have multiplicity exactly one, where $0 \leq{m}\leq{3}$. Then $3 - m$ bounding lines have multiplicity zero. A bounding line contains $n-2$ points that are not covered by other bounding lines. In the union of the faces of $T_2(n)$, there are at least $m(n-2)$ points that need to be covered one additional time and $(3-m)(n-2)$ points that need to be covered two additional times. Since no line (that is not a bounding line) can cover more than two of these points, we obtain
$$ |\mathcal{C}| \geq m + \frac{m(n-2) + 2(3-m)(n-2)}{2} = \frac{3n}{2}+\frac{(3-m)(n-4)}{2}.$$
Since $|\mathcal{C}|$ is an integer, we obtain $f(n, 2, 2) \geq \ceil{3n/2}$.

\emph{The case $k = 4$:} 
Recall that we want to show $f(n, 2, 4) \geq 3n$. For $n\in\{3,4,5\}$, the result follows immediately from~\eqref{eq:fractionaltointegerbound}. By Observation \ref{boundary}, it suffices to consider 4-covers where the multiplicity of each bounding line is at most two.
Let $n\geq{6}$ and $\mathcal{C}$ be a 4-cover of $T_2(n)$ where $m$ bounding lines have multiplicity exactly two, where $0\leq{m}\leq{3}$. Then $3-m$ bounding lines have multiplicity at most one. In the union of the faces of $T_2(n)$, there are at least $m(n-2)$ points that need to be covered two additional times and $(3-m)(n-2)$ points that need to be covered at least three additional times.
Since no line (that is not a bounding line) can cover more than two of these points, we obtain 
$$ |\mathcal{C}| \geq {2m + \frac{2m(n -2)+3(3-m)(n-2)}{2}} = 3n+\frac{(3-m)(n-6)}{2},$$
implying that $f(n,2,4)\geq{3n}$. 

\emph{The case $k=3$:}
Let $s$ (resp. $t$) be the smallest nonnegative integer such that the line given by $x = s$ (resp. $x = t$) has multiplicity at most one (resp. has multiplicity zero). Then the number of vertical lines (with multiplicity) is at least $2s + (t-s) = s+t$. Each of the $n-s$ points on the line given by $x = s$ must be covered at least an additional two times, and any non-vertical line can cover exactly one of these. It follows that
\begin{equation}
\label{eq:d2k3eq1}
|\mathcal{C}| \geq s + t + 2(n-s) = 2n - s + t.
\end{equation}
Similarly, each of the $n-t$ points on the line given by $x = t$ must be covered three times, and any non-vertical line can cover exactly one of these. This gives 
\begin{equation}
\label{eq:d2k3eq2}
|\mathcal{C}| \geq s + t + 3(n-t) = 3n+s - 2t.
\end{equation}
Combining \eqref{eq:d2k3eq1} and \eqref{eq:d2k3eq2}, we obtain
$$|\mathcal{C}| \geq \frac{7n - s}{3} \quad \text{ and }\quad |\mathcal{C}| \geq \frac{5n-t}{2}.$$
Suppose, for contradiction, that 
$|\mathcal{C}| < 9n/4$. Then $s > n/4$ and $t > n/2$, implying that there are more than $3n/4$ vertical lines (with multiplicity) in $\mathcal{C}$. By the same argument, the number of horizontal and diagonal standard lines is also more than $3n/4$. But then $|\mathcal{C}| > 3(3n/4) = 9n/4$, a contradiction. Since $|\mathcal{C}|$ is an integer, we obtain $f(n, 2, 3) \geq \ceil{9n/4}$.

We now present constructions which imply corresponding upper bounds. 

\emph{The case $k = 1$:} A cover of size $n$ is easily obtained by considering the lines given by  $x = i$, for $i\in\{0,\dots,n-1\}$.

\emph{The case $k = 2$:} Our construction depends on the parity of $n$.

For even $n$, the cover consists of the lines given by $x = i$, $y = i$, and $x+y = n - 1 - i$, for $i \in \{0,\dots, n/2-1 \}$.

For odd $n$, the cover consists of the lines given by $x = i$, $y = i$,  for $i \in \{0, \dots, (n-1)/2 \}$ and $x+y = n - 1 - i$, for $i \in \{0, \dots, (n-3)/2 \}$.

\emph{The case $k = 3$:}  The cover contains:
\begin{itemize}
    \item with multiplicity two, the lines given by $x=i$ and $y = i$, for $i\in\{0,\dots,\floor{\frac{n-2}{4}}\}$,
    \item with multiplicity one, the lines given by $x=i$ and $y = i$, for $i\in\{\floor{\frac{n-2}{4}}+1,\dots,   \floor{\frac{n-1}{2}}\}$,
    \item with multiplicity two, the lines given by $x+y=n-1-i$ for $i\in\{0,\dots,\floor{\frac{n}{4}}-1\}$,
    \item with multiplicity one, the lines given by $x+y=n-1-i$ for
    $$ i \in
    \begin{cases}
    \{\floor{\frac{n}{4}},\dots,\floor{\frac{n}{2}}\} & n\equiv 1\pmod{4},\\
    \{\floor{\frac{n}{4}},\dots,\floor{\frac{n}{2}}-1\} & \text{otherwise}.
    \end{cases}$$
\end{itemize}

\emph{The case $k = 4$:} 
The cover contains:
\begin{itemize}
    \item with multiplicity two, the lines given by $x=i, y = i$, and $x+y=n-1-i$ for $i\in\{0,\dots,\floor{\frac{n-1}{3}}\}$,
    \item with multiplicity one, the lines given by $x=i, y = i$, and $x+y=n-1-i$ for $i\in\{\floor{\frac{n-1}{3}}+1,\dots,\floor{\frac{2n}{3}}-1\}$.
\end{itemize}
Here, we verify the upper bound when $k=4$. We omit the verification for other values of $k$, which proceed similarly.

In each of the three standard directions, there are $\floor{\frac{n-1}{3}}+1$ lines with multiplicity two and ${\floor{\frac{2n}{3}}-\floor{\frac{n-1}{3}}-1}$ lines with multiplicity one. Note that this is still true for $n=2, 4$, where ${\floor{\frac{2n}{3}}-1<\floor{\frac{n-1}{3}}+1}$ and no line has multiplicity exactly one. Hence the total number of lines is
$$3\left[2\left(\floor{\frac{n-1}{3}}+1\right)+\left(\floor{\frac{2n}{3}}-\floor{\frac{n-1}{3}}-1\right)\right]=3\left(\floor{\frac{2n}{3}}+\floor{\frac{n-1}{3}}+1\right)=3n.$$

It remains to verify that each point $(a,b)\in T_2(n)$ is covered at least four times. Suppose, without loss of generality, that $a\leq{b}$. 

If $a\geq{\floor{\frac{n-1}{3}}+1}$ and $b>\floor{\frac{2n}{3}}-1$, then $a+b>n-1$, a contradiction.

If $\floor{\frac{n-1}{3}}+1\leq{a,b}\leq{\floor{\frac{2n}{3}}-1}$, then $(a,b)$ is covered twice by vertical and horizontal lines. Note also that $a+b\geq{2\floor{\frac{n-1}{3}}+2}\geq{n-1-\floor{\frac{n-1}{3}}}$, so $(a,b)$ is covered an additional two times by diagonal lines.

Suppose now that $a\leq{\floor{\frac{n-1}{3}}}$. If $b\leq{\floor{\frac{n-1}{3}}}$, then $(a,b)$ is covered four times by vertical or horizontal lines. If $\floor{\frac{n-1}{3}}+1\leq{b}\leq{\floor{\frac{2n}{3}}}-1$, then $(a,b)$ is covered three times by vertical or horizontal lines, and once by a diagonal line since $a+b\geq{\floor{\frac{n-1}{3}}+1}\geq{n-\floor{\frac{2n}{3}}}$. Finally, if $b > \floor{\frac{2n}{3}}-1$, then $(a,b)$ is covered twice by vertical or horizontal lines, and twice by a diagonal line since  $a+b\geq{\floor{\frac{2n}{3}}}\geq{n-1-\floor{\frac{n-1}{3}}}$.
Thus, all points are covered at least four times.
\end{proof}

The values determined in Theorem \ref{2int} along with computational evidence for $k \leq 7$ suggest a connection between the fractional covering problem for $T_2(k)$ and the integer $k$-covering problem for a triangular grid of any size:
\begin{conjecture}\label{duality}
$f(n,2,k)=f^*(k,2)n+O_k(1)$.
\end{conjecture}

For the remainder of this section, we provide a partial resolution to Conjecture~\ref{duality}. In Proposition~\ref{upper}, we show that the upper bound on $f(n,2,k)$ always holds. In Theorem~\ref{partialpr}, we prove the lower bound holds under certain natural assumptions on the lines contained in the $k$-cover.
\begin{proposition}\label{upper}
$f(n,2,k) \leq f^*(k,2)n + O_k(1)$.
\end{proposition}
\begin{proof}
When $k=1$, the assertion follows from Theorem \ref{2int}. Hence, we assume that $k > 1$. 

We first give an overview of the proof idea. The $k$-cover $\mathcal{C}$ of $T_2(n)$ given below is obtained by \emph{lifting} the fractional covers given in the proofs of Theorems~\ref{1mod3} and \ref{20mod3}. More specifically, we group the standard lines of $T_2(n)$ into sets $B_1, \dots, B_N$, each consisting of $\ceil{\frac{n}{M}}$ consecutive standard lines in each direction (the values of $N$ and $M$ are specified below). Roughly, each set $B_i$ corresponds to a standard line of $T_2(k)$ with positive weight in the fractional cover. This weight determines the multiplicity in $\mathcal{C}$ of each line in $B_i$ (so all lines in $B_i$ have the same multiplicity). 

Let $j = \floor{(k-1)/3}$, so that $k = 3j + r$ with $r \in \{1, 2, 3\}$. Set $N := k-j-1$ and $M := 2k-3j-2$. Note that $N, M > 0$ since $k > 1$.
For each $u \in \{1, \dots, N\}$, let $B_u$ be the set of lines given by $x = i, y = i$, and $x + y = n - 1 - i$ with $i \in \{(u-1)\ceil{\frac{n}{M}}, \dots,  u\ceil{\frac{n}{M}} - 1\}$. When $n$ is sufficiently large compared to $k$, the sets $B_u$ contain only standard lines. To explain the choice of $N$ and $M$, we note that
$$ N = \begin{dcases}
2j & \mbox{if } r = 1,\\
2j + 1 & \mbox{if } r = 2,\\
2j + 2 & \mbox{if } r = 3,\\
\end{dcases} \quad \text{and} \quad M = \begin{dcases}
3j & \mbox{if } r = 1,\\
3j + 2 & \mbox{if } r = 2,\\
3j + 4 & \mbox{if } r = 3.\\
\end{dcases} $$
So, $N$ is precisely the number of standard lines in each direction with positive weight in the fractional covers given in the proofs of Theorems~\ref{1mod3} and \ref{20mod3}, and $M$ is the denominator in the weight assigned to each line in the same fractional cover.

Consider the $k$-cover $\mathcal{C}$ which contains, for each $u$, the lines in $B_u$ with multiplicity $k-j-u$, and no other line with positive multiplicity. Note that $\mathcal{C}$ contains $3N\ceil{\frac{n}{M}}$ distinct lines with an average multiplicity of $(k-j)/2$, implying that
\begin{equation}
\label{eq:conjupperbound}
|\mathcal{C}| = 3N\left(\frac{n}{M} + O(1)\right)\left(\frac{k-j}{2}\right) = \frac{3N(k-j)}{2M} n + O_k(1).
\end{equation}
Noting that
$$ \frac{3N(k-j)}{2M} = \begin{dcases}
2j + 1 & \mbox{if } r = 1,\\
2j + 1 + \frac{2j+1}{3j+2} & \mbox{if } r = 2,\\
2j + 2 + \frac{j+1}{3j+4} & \mbox{if } r = 3,\\
\end{dcases} $$
we obtain
$$|\mathcal{C}| = f^*(k, 2)n + O_k(1).$$

It remains to show that every point of $T_2(n)$ is covered at least $k$ times. Let $p = (a, b) \in T_2(n)$. Note that at least one of the lines given by $x = a$ and $y = b$ must have positive multiplicity in $\mathcal{C}$, otherwise $a + b \geq (2k - 2j - 2) \ceil{\frac{n}{M}} \geq n$, a contradiction. By symmetry, the same is true for any pair of standard lines containing $p$. Hence, we may assume, without loss of generality, that the lines given by $x = a$ and $y = b$ are contained in the sets $B_s$ and $B_t$ respectively. 

Now $p$ is covered with multiplicity $2k - 2j - (s + t)$ by vertical and horizontal lines.
Hence, if $s+t \leq k - 2j$, then $p$ is covered with multiplicity at least $k$.  From here on, we assume that $s + t > k - 2j$, so $s + t \geq k - 2j+1$. Note that
$$ (s + t - 2) \ceil{\frac{n}{M}} \leq a + b \leq n - 1, $$
which implies
\begin{align*}
0 \leq n-1-(a+b) & \leq n-1 - (s + t - 2) \ceil{\frac{n}{M}} \leq (M - (s + t - 2))\ceil{\frac{n}{M}} - 1.
\end{align*}
Noting that 
$$M-(s+t-2) = 2k-3j - (s+t) \leq k - j - 1 = N,$$
we obtain that the line given by $x + y  = a + b$ is contained in the set $B_r$ for some $r \le M-(s+t-2)$, and, hence, is contained in the cover $\mathcal{C}$ with multiplicity $k-j-r$ which is at least
$$k - j - (2k - 3j - (s + t)) = 2j - k + s + t.$$
Since $p$ is contained in the line $x + y = a + b$, we obtain that $p$ is covered with total multiplicity at least~$k$.
\end{proof}

In Theorem \ref{partialpr}, we provide a partial resolution of the lower bound in Conjecture \ref{duality}. Note that all the covers we have given in order to prove upper bounds consist only of standard lines.
Furthermore, for two parallel lines, the multiplicity of the line containing more points of $T_d(n)$ is at least the multiplicity of the other line. The following theorem proves Conjecture \ref{duality} for such covers.

\begin{theorem}\label{partialpr}
Suppose $\mathcal{C}$ is a $k$-cover of $T_2(n)$ satisfying the following:
\begin{itemize}
\item any line contained in $\mathcal{C}$ is a standard line, i.e., a line parallel to a side of $T_2(n)$,
\item if $\ell_1, \ell_2 \in \mathcal{C}$ are parallel lines with $|\ell_1 \cap T_2(n)| \leq |\ell_2 \cap T_2(n)|$, then the multiplicity of $\ell_1$ is at most the multiplicity of $\ell_2$.
\end{itemize}
Then $|\mathcal{C}| \geq f^*(k,2)n + O_k(1)$.
\end{theorem}
\begin{proof}
Our proof proceeds by explicitly comparing the linear program for fractionally covering $T_2(k)$ and the integer program for $k$-covering $T_2(n)$ for any $n$.

We first describe the linear program for fractionally covering $T_2(k)$. We will refer to this linear program as $\mathbf{L}$. For each $i \in \{1, \dots, k-1\}$, we associate variables $a_i$, $b_i$, $c_i$ to the standard lines given by $x= k - 1 -i$, $y = k - 1 -i$, $x + y = i$. Observe that the lines corresponding to $a_i, b_i,$ and $c_i$ each contain $i+1$ points.
The variables give the weights of the corresponding lines in a fractional cover. Note that we do not associate a variable to any line that contains at most one point of $T_2(k)$, e.g., the line $x = k-1$. This is without loss of generality, as there is an optimal cover with these lines having weight 0. We also have (a finite number of) additional variables $d_1, \dots, d_q$ corresponding to non-standard lines that contain at least two points of $T_2(k)$. With these definitions, the objective of the linear program is to minimize the sum of all variables, i.e., the objective function is
$$ \text{minimize} \quad \sum_{i = 1}^{k-1} (a_i + b_i + c_i) + \sum_{i = 1}^{q} d_i. $$

The constraints for the linear program consist of a non-negativity constraint for each variable, and a constraint that each point of $T_2(k)$ is contained in lines of total weight at least 1. Specifically, for each point $p \in T_2(k)$, we have a constraint of the form 
\begin{equation}
\label{eq:lpconstraints}
a_r + b_s + c_t + \sum_{i \in \mathcal{L}(p)} d_i \geq 1,
\end{equation}
where $\mathcal{L}(p)$ is the set of indices of non-standard lines containing $p$ (and at least one other point of $T_2(k)$). Note that three standard lines intersect in a point of $T_2(k)$ if and only if they contain a total of $2k+1$ points of $T_2(k)$ (with multiplicity).  This implies that we have such a constraint if and only if $(r+1) + (s+1) + (t + 1) = 2k+1$, or, equivalently, $r + s + t = 2k - 2$.

We now consider the restricted $k$-covering problem for $T_2(n)$ and describe the corresponding integer program. We will refer to this integer program as $\mathbf{I}$. Recall that we require the covers to satisfy the following restrictions:
\begin{itemize}
\item any line contained in the cover is a standard line,
\item if $\ell_1, \ell_2$ are parallel lines with $|\ell_1 \cap T_2(n)| \leq |\ell_2 \cap T_2(n)|$, then the multiplicity of $\ell_1$ is at most the multiplicity of $\ell_2$.
\end{itemize}

We may assume that every line has multiplicity less than $k$. Indeed, if some line has multiplicity at least $k$, then there is a bounding line with multiplicity at least $k$. However, by Observation~\ref{boundary}, we  may assume that each bounding line has multiplicity less than $f^*(k,2)$, and $f^*(k,2) \leq k$ by Theorems~\ref{1mod3} and \ref{20mod3}.

For $i\in\{1,\dots,k-1\}$, let $\alpha_i$ denote the number of vertical lines which have multiplicity at least $k - i$. Similarly, let $\beta_i$ and $\gamma_i$, with $i \in \{1, \dots, k-1\}$, denote the number of horizontal and diagonal standard lines (respectively) which have multiplicity at least $k - i$. Recall that the goal is to find a minimum $k$-cover. The number of vertical lines (with multiplicity) in the cover is
$$(k-1)\alpha_1 + (k-2)(\alpha_2 - \alpha_1) + (k - 3)(\alpha_3 - \alpha_2) + \dots + 1(\alpha_{k-1} - \alpha_{k-2}) = \sum_{i = 1}^{k-1}\alpha_i.$$
Similarly, the number of vertical and diagonal lines is $\sum_{i = 1}^{k-1} \beta_i$ and $\sum_{i = 1}^{k-1} \gamma_i$ respectively. Hence, the objective function is 
$$ \text{minimize} \quad \sum_{i = 1}^{k-1} (\alpha_i + \beta_i + \gamma_i).$$

The constraints for the integer program consist of an integrality and non-negativity constraint for each variable. We also need constraints specifying that each point $p \in T_2(n)$ is covered by at least $k$ lines.

Consider an assignment of values to $\alpha_i$'s, $\beta_i$'s, and $\gamma_i$'s such that each point $p$ is covered at least $k$ times. Suppose $\alpha_r+\beta_s\le n-1$ and let  $p=(\alpha_r,\beta_s)\in T_2(n)$. Since $p$ is covered at most $(k-r)-1+(k-s)-1=2k-(r+s+2)$ times by horizontal and vertical lines, it must be covered at least an additional $k-[2k-(r+s+2)]=r+s+2-k$ times by the diagonal line $x+y=\alpha_r+\beta_s$. This implies that if $r+s+2-k\in\{1,\dots,k-1\}$, then
$\gamma_{k-(r+s+2-k)}\geq{n-(\alpha_r+\beta_s)}$, that is, 
$$\alpha_r+\beta_s+\gamma_{2k-2-(r+s)}\ge n.$$
Note that if $r+s+t=2k-2$ with $r,s,t\in\{1,\dots,k-1\}$, then $r+s+2-k\in\{1,\dots,k-1\}$. Hence, the following constraints are necessary to ensure that each point of $T_2(n)$ is covered $k$ times. For each $r, s, t$ such that $r + s + t = 2k - 2$, we have the constraint
\begin{equation}
\label{eq:ipconstraints}
\alpha_r+\beta_s+\gamma_t\ge n.
\end{equation}

Suppose we have a feasible solution to the integer program $\mathbf{I}$, and let $B := \sum_{i = 1}^{k-1} (\alpha_i + \beta_i + \gamma_i)$. Of course, this feasible solution corresponds to a $k$-cover $\mathcal{C}$ of $T_2(n)$ with $|\mathcal{C}| = B$ satisfying the assumptions of the theorem. Recall that we want to prove $B \geq f^*(k, 2)n + O_k(1)$. We show this by constructing a feasible solution to the linear program $\mathbf{L}$ such that the objective function $\sum_{i = 1}^{k-1} (a_i + b_i + c_i) + \sum_{i = 1}^{q} d_i \leq B/n$. Since $\mathbf{L}$ is a minimization problem, this implies the result.

The feasible solution to $\mathbf{L}$ is obtained as follows: for each $i \in \{1, \dots, k-1\}$, set $a_i = \alpha_i/n$, $b_i = \beta_i/n$ and $c_i = \gamma_i/n$. For each $i \in\{1, \dots, q\}$, set $d_i = 0$.  That this results in a feasible solution is easy. Note first that all the non-negativity constraints are satisfied. To see that the constraints given by~\eqref{eq:lpconstraints} are satisfied, it suffices to compare to the constraints given by~\eqref{eq:ipconstraints} and to note that $d_i = 0$ for each $i$.

\end{proof}

\section{Higher dimensions}\label{sec:highdim}
In this section, we extend the results and techniques of Section~\ref{sec:intcovering} to the higher-dimensional setting. We first consider the problem of covering the points of $T_{d}(n)$ at least $k$ times for $k\le 3$.
\begin{theorem}\label{prop:1ce2ce}
For every $n\geq{1}$,
\begin{enumerate}[label = (\alph*)]
    \item \label{itm:highdim1} For $d\geq{1}$, $f(n,d) = n$,
    \item \label{itm:highdim2} For $d\geq{1}$, $f(n,d,2) = n+ \lceil{\frac{n}{d}\rceil}$,
    \item \label{itm:highdim3} For $d\geq{3}, f(n,d,3) = \left(1+\frac{2}{d-1}\right)n + O_d(1)$.
\end{enumerate}
\end{theorem}

\begin{obs}\label{d-2}
    The intersection of a bounding hyperplane $H$ with $T_d(n)$ is a copy of $T_{d-1}(n)$. Any  hyperplane not parallel to $H$ intersects $H$ in a $(d-2)$-dimensional affine subspace. Hence covering the points of $H\cap T_d(n)$ at least $k$ times without using $H$ requires at least $f(n,d-1,k)$ hyperplanes.
\end{obs}

\begin{proof}[Proof of Theorem \ref{prop:1ce2ce}] {~}

\emph{Proof of \ref{itm:highdim1}:} We begin by proving the lower bound. It is straightforward that $f(1,d)\ge 1$. Let  $\mathcal{C}$ be a cover of $T_d(n)$ and $H$ be the hyperplane given by $x_1 = 0$. By Observation~\ref{boundary}, we may assume that $H \notin \mathcal{C}$.
We have $f(n,1)=n$, and Observation~\ref{d-2} gives $|\mathcal{C}|\geq{f(n,d-1)}$ so by induction on $d$, $f(n,d)\geq{n}$. 

A matching upper bound can be obtained by considering the hyperplanes, $x_1 = i$ for ${i\in\{0, \dots, n-1\}}$.

\emph{Proof of \ref{itm:highdim2}:} The proof of the lower bound in \ref{itm:highdim2} also proceeds by induction on $d$ and $n$. When $n=1$, $T_d(1)$ consists of a single point and a cover requires two hyperplanes. For $d=1$, each hyperplane can cover at most one point of $T_1(n)$, so $2n$ planes are required. Suppose now that $d, n > 1$.

Let  $\mathcal{C}$ be a 2-cover of $T_d(n)$, $H_0$ be the hyperplane given by $x_1+\dots+x_d=n-1$, and $H_i$ be the hyperplane given by $x_i = 0$, for $i \in \{1, \dots, d\}$. By Observation \ref{boundary}, we may assume that the multiplicity of each $H_i$ in $\mathcal{C}$ is at most one.

If, for some $i$, $H_i \notin \mathcal{C}$, Observation \ref{d-2} implies that
$$|\mathcal{C}| \geq f(n, d-1, 2) = n + \ceil{\frac{n}{d-1}} \geq {n + \ceil{\frac{n}{d}}}.$$

Thus, we may assume that the multiplicity of each $H_i$ in $\mathcal{C}$ is exactly one. Note that $H_i \cap T_{d}(n)$ constitutes a copy of $T_{d-1}(n)$ where only the points on the boundary are covered at least twice, leaving a copy of $T_{d-1}(n-d)$ that has only been covered once. Each of these requires $f(n-d, d-1, 1) = n-d$ affine subspaces of dimension $d-2$ to be covered. A hyperplane $H \in \mathcal{C}\setminus \bigcup_{i= 0}^{d} \{ H_i\}$ cannot intersect every face of $T_d(n)$ so it can contribute at most $d$ such $(d-2)$-dimensional subspaces, implying that 
$$
    |\mathcal{C}| \geq (d+1) + \ceil{\frac{(d+1)(n-d)}{d}} = n+ \ceil{\frac{n}{d}}.
$$
This completes the proof of the lower bound in \ref{itm:highdim2}.

We now establish the upper bound in \ref{itm:highdim2}. Let $q \in \mathbb{Z}$ and $0\leq{r} < d+1$ be such that $n + \ceil{\frac{n}{d}} = q(d+1)+r$. The cover consists of the following:
\begin{itemize}
\item the hyperplane given by $x_i = j$ for $1 \leq i \leq r$ and $j\in\{0,\dots,q\}$,
\item the hyperplane given by $x_i = j$ for $r+1 \leq i \leq d$ and $j\in \{0,\dots, q-1\}$,
\item the hyperplane given by $x_1+\dots+x_d = n -1 - j$ for $j\in\{0,\dots,q-1\}$.
\end{itemize}
If $1 \leq i \leq r$, we set $c_i = q$, and if $r+1 \leq i \leq d$, we set $c_i = q-1$. Let $p = (a_1, \dots, a_d) \in T_d(n)$. There must be at least one index $i$ such that $a_i \leq c_i$. Indeed, otherwise
$$  \sum_{i = 1}^{d} a_i \geq r(q+1)+(d-r)q = qd+r\geq{\frac{d}{d+1}\left(n+\ceil{\frac{n}{d}}\right)} \geq n, $$
a contradiction. If there are at least two indices such that $a_i \leq {c_i}$, then $p$ is covered at least twice by hyperplanes orthogonal to standard basis vectors.

Suppose that there is exactly one index~$i$ such that $a_i \leq c_i$. Note that $p$ is covered by the hyperplane $x_i = a_i$. It suffices to show that $\sum_{i = 1}^{d} a_i \geq n - q$, which implies that $p$ is covered by $x_1+\dots+x_d = \sum_{i = 1}^{d} a_i$.

Since $n + \ceil{\frac{n}{d}} = q(d+1) + r$, we have
\begin{align}
\label{eq:highdimupperboundeq}
n - q \leq q(d - 1) + r\left(1 - \frac{1}{d+1}\right).
\end{align}
If $r = 0$, then $\sum_{i = 1}^{d} a_i \geq (d-1)q \geq n - q$. If $r > 0$, then $\sum_{i = 1}^{d} a_i \geq (r-1)(q+1) + (d-r)q = q(d - 1) + r - 1$. Since $0 < r/(d+1) < 1$ and $n - r$ is an integer, we obtain, by \eqref{eq:highdimupperboundeq}, that $\sum_{i = 1}^{d} a_i \geq n - q$. 

\emph{Proof of \ref{itm:highdim3}:} We first show that $f(n, d, 3) \geq \left(1+\frac{2}{d-1}\right)n-2$. We begin by showing $f(n,3,3)\geq{2n-2}$ via induction on $n$.
 The bound is trivial for $n = 1$. Suppose now that $n > 1$ and that the assertion is true for smaller $n$. 
 
Let  $\mathcal{C}$ be a 3-cover of $T_d(n)$, $H_0$ be the hyperplane given by $x_1+x_2+x_3=n-1$, and $H_i$ be the hyperplane given by $x_i = 0$, for $i \in \{1, 2, 3\}$. By Observation \ref{boundary}, we may assume that the multiplicity of each $H_i$ in $\mathcal{C}$ is at most one.
If, for some $i$, $H_i \notin \mathcal{C}$, then by Observation \ref{d-2}, $|\mathcal{C}| \geq f(n, 2, 3)=\ceil{\frac{9n}{4}}>2n-2$.
 
We may now assume that, for each $i$, the hyperplane $H_i$ has multiplicity exactly one.
Observe that $H_i \cap T_3(n)$ is a copy of $T_2(n)$ where only the boundary points have been covered at least twice. The interior points constitute a copy of $T_2(n-3)$ which has only been covered once and, by Theorem~\ref{2int}, requires at least $f(n-3, 2, 2) = \ceil{\frac{3(n-3)}{2}}$ lines to be covered three times. Since every hyperplane in $\mathcal{C}$ that is not a bounding hyperplane can intersect at most three faces of $T_3(n)$, we obtain that $$|\mathcal{C}| \geq 4 + \frac{4}{3}\ceil{\frac{3(n-3)}{2}} \geq 2n - 2.$$ This completes the proof that $f(n,3,3)\ge 2n-2$.

Now suppose $d>3$ and proceed by induction on $d$. We trivially have $f(n,d,3)\geq{(1+\frac{2}{d-1})n-2}$ for $n=1$ and by Observation \ref{boundary}, we may assume that each bounding hyperplane is used at most once.
Let $\mathcal{C}$ be a 3-cover of $T_d(n)$ and $H$ be a bounding hyperplane of $T_d(n)$. If $H \notin \mathcal{C}$, then Observation \ref{d-2} gives $$|\mathcal{C}| \geq f(n,d-1,3)\ge \left(1+\frac{2}{d-2}\right)n-2 \geq \left(1+\frac{2}{d-1}\right)n-2. $$
    
Thus, we may assume the multiplicity in $\mathcal{C}$ of each bounding hyperplane is exactly one. Let $s$ be the smallest nonnegative integer such that $x_1 = s$ has multiplicity zero. Such an $s$ exists and, furthermore, $s \leq n-1$ 
since, without loss of generality, the hyperplane given by $x_1 = n-1$ has multiplicity zero.
Noting that the points in $T_d(n)\cap H$ still need to be covered twice, we obtain
\begin{equation}\label{eq:highdim3cover1}
    |\mathcal{C}| \geq s + f(n, d-1, 2) \geq s + n + \ceil{\frac{n}{d-1}} \geq s + n\left(1 +  \frac{1}{d-1}\right).
\end{equation} 
On the other hand, if $H_s'$ is the hyperplane given by $x_1 = s$, then the points in $T_d(n) \cap H_s'$ need to be covered at least three times, implying
\begin{equation}\label{eq:highdim3cover2}
    |\mathcal{C}| \geq s + f(n - s, d-1, 3) \geq s + \left(1 + \frac{2}{d-2}\right)(n-s) - 2 = n\left(1 + \frac{2}{d-2}\right) - s\left(\frac{2}{d-2}\right)  - 2.
\end{equation} 
Multiplying (\ref{eq:highdim3cover1}) by $\frac{2}{d-2}$ and adding it to (\ref{eq:highdim3cover2}) eliminates $s$, and the inequality simplifies to $$|\mathcal{C}| \geq \left(1+\frac{2}{d-1}\right)n - 2\left(\frac{d-2}{d}\right)\geq{\left(1+\frac{2}{d-1}\right)n - 2}. $$

We now establish the upper bound in \ref{itm:highdim3}. The 3-cover consists of the following:
\begin{itemize}
\item the hyperplane given by $x_i = j$, for $1 \leq i \leq d$ and $j\in\{0,\dots,\ceil{\frac{n}{d-1}}-1\}$,
\item the hyperplane given by $x_1+\dots+x_d = n -1 - j$, for $j\in\{0,\dots,\ceil{\frac{n}{d-1}}-1\}$.
\end{itemize}
This gives a total of $(d+1)\ceil{\frac{n}{d-1}} = \left(1+\frac{2}{d-1}\right)n + O_d(1)$ hyperplanes. Let $p = (a_1, \dots, a_d) \in T_d(n)$.
We say a coordinate of $p$ is \emph{small} if it is at most $\ceil{\frac{n}{d-1}}-1$.
Note that $p$ has at least two small coordinates. Indeed, otherwise $\sum_{i = 1}^{d} a_i \geq (d-1) \ceil{\frac{n}{d-1}} \geq n$, a contradiction.
 If $p$ has at least three small coordinates, then $p$ is covered at least three times by hyperplanes orthogonal to the standard basis vectors. Suppose that $p$ has exactly two small coordinates. Then $$\sum_{i = 1}^{d} a_i \geq (d-2)\ceil{\frac{n}{d-1}} \geq n - \ceil{\frac{n}{d-1}}.$$ It follows that $p$ is covered by $x_1+\dots+x_d=\sum_{i=1}^d a_i$. Since $p$ is also covered twice by hyperplanes orthogonal to the standard basis vectors, $p$ is covered by at least three hyperplanes in the cover, thus completing the proof.
\end{proof}

Given Conjecture~\ref{duality} and the strong connection between the fractional and integer problems in dimension two, we also investigate the fractional problem in higher dimensions.  

\begin{proposition}
\label{prop:f2d}
For every $d > 0$,
\[f^*(2,d)=1+\frac{1}{d}.\]
\end{proposition}

\begin{proof}
Note that the $d+1$ points of $T_d(2)$ are the vertices of a $d$-dimensional simplex. For the upper bound, assigning weight $\frac{1}{d}$ to each of the $d+1$ hyperplanes determined by a $d$-subset of points gives a fractional cover of weight $1+\frac{1}{d}$.

For the lower bound, we assign each point in $T_d(2)$ a mass of $\frac{1}{d}$, for a total mass of $1+\frac{1}{d}$. Since any hyperplane can cover points of total mass at most $1$, we obtain the required lower bound.
\end{proof}

From Theorem~\ref{prop:1ce2ce}\ref{itm:highdim2} and Proposition \ref{prop:f2d}, we obtain $$f(n,d,2)=(1+1/d)n+O(1)=f^*(2,d)n+O(1).$$ Thus, it is tempting to suggest a more general version of Conjecture~\ref{duality}: that $f(n,d,k) = f^*(k,d)n + O_{d,k}(1)$. However, the following proposition, along with Theorem~\ref{prop:1ce2ce}~\ref{itm:highdim3}, implies that $f(n,3,3)> f^*(3,3)n+O(1)$.
\begin{proposition}\label{conjfails}
    $f^*(3,3) \leq \frac{11}{6}$.
\end{proposition}
\begin{proof}
    The following construction implies an upper bound of $\frac{11}{6}$ for $f^*(3, 3)$:
    \begin{itemize}
    \item the planes $x_1=0, x_2=0, x_3=0$, and $x_1+x_2+x_3=2$ have weight $\frac{1}{3}$,
    \item the planes $x_1+x_2 = 1, x_1 + x_3 = 1,$ and $x_2+x_3=1$ have weight $\frac{1}{6}$.
    \end{itemize}
    It is straightforward to check that each point of $T_3(3)$ is contained in planes with total weight at least 1.
\end{proof}

We are, however, able to obtain an upper bound on $f(n, d, k)$ for all $d, k$. The construction is inspired by features of the various upper bound constructions so far. In particular, the cover consists only of standard hyperplanes. If $h_1, h_2$ are parallel standard hyperplanes with $|h_1 \cap T_d(n)| \leq |h_2 \cap T_d(n)|$, then the multiplicity of $h_1$ is at most the multiplicity of $h_2$. Furthermore, the hyperplanes in each direction can be partitioned into contiguous blocks of equal size, where the hyperplanes in each block have the same multiplicity. 
\begin{proposition}\label{upboundall}
   Suppose $d$ is odd and $k= (d+1)q/2 + r$ with $q, r \in \ZZ$ and $1\leq r \leq (d+1)/2$. Then $f(n,d,k)\leq{C_{d,k}n}+O_{d,k}(1)$ where 
   \[ C_{d,k}=(q+1)\left(1+\frac{r-1}{(d+1)(q+2)/2 - (r-1)}\right). \]
   Suppose $d$ is even and $k=(d+1)q+r$ with $q, r \in \ZZ$ and $1\leq r\leq {d+1}$. Then $f(n,d,k)\leq{C_{d,k}n}+O_{d,k}(1)$ where
$$ C_{d,k}=\begin{dcases} (2q+1)\left(1 +\frac{r-1}{(d+1)(q+1) - (r-1)} \right)&\mbox{if }1\leq{r}\leq{d/2+1}\mbox{,}\\
(2q+3)\left(1-\frac{d+2-r}{(d+1)(q+1)+(d+2-r)}\right)&\mbox{otherwise.}
\end{dcases}$$
\end{proposition}
\begin{proof}
Suppose $d$ is odd and $k= (d+1)q/2 + r$ with $q, r \in \ZZ$ and $1\leq r \leq (d+1)/2$. Set $M := (d+1)(q+2)/2 - (r-1)$. For each $j \in \{1, \dots, q+1\}$, let $B_j$ consist of the following :
\begin{itemize}
\item hyperplanes given by $x_i = c$, for $1 \leq i \leq d$ and $c\in\{(j-1)\ceil{\frac{n}{M}},\dots,j\ceil{\frac{n}{M}} - 1\}$,
\item hyperplanes given by $x_1+\dots+x_d = n -1 - c$, for $c\in\{(j-1)\ceil{\frac{n}{M}},\dots,j\ceil{\frac{n}{M}} - 1\}$.
\end{itemize}
Note that, for $n$ large enough, each set $B_j$ consists of $\ceil{\frac{n}{M}}$ standard hyperplanes in each of $d+1$ directions. Consider the $k$-cover $\mathcal{C}$ which contains, for each $j$, the hyperplanes in $B_j$ with multiplicity $q + 2 - j$, and no other hyperplane with positive multiplicity. Hence 
\begin{align*}
|\mathcal{C}| &= (d+1)\ceil{\frac{n}{M}}\sum_{j = 1}^{q+1} j = (d+1)\left(\frac{n}{M}+O(1)\right)\left(\frac{(q+1)(q+2)}{2}\right) \\
& = (q+1)\left(\frac{(q+2)(d+1)}{2M}\right)n +O_{d, k}(1) = (q+1)\left(1 + \frac{r-1}{M}\right)n +O_{d, k}(1).
\end{align*}

It remains to verify $\mathcal{C}$ is a $k$-cover of $T_d(n)$. Let $p = (a_1, \dots, a_d) \in T_d(n)$ and suppose that, for each~$i$, $p$ is covered $c_i$ times by the hyperplane given by $x_i = a_i$. Then
\[
\sum_{i = 1}^{d}a_i \geq \ceil{\frac{n}{M}}\left(d(q+1)-\sum_{i=1}^d c_i\right).
\]
If $\sum_{i=1}^d c_i\geq {k}$, then $p$ is covered at least $k$ times. On the other hand, if $\sum_{i=1}^d c_i < k- (q+1)$, then
$$\sum_{i = 1}^{d}a_i \geq \ceil{\frac{n}{M}}\left(d(q+1)- (k - q - 2) \right) = \ceil{\frac{n}{M}} M \geq n, $$
a contradiction.
We may now assume that $\sum_{i=1}^d c_i = k-a$ with $1\leq{a}\leq{q+1}$. Then 
\begin{align*}
\sum_{i = 1}^{d}a_i \geq \ceil{\frac{n}{M}}\left(d(q+1) - k + a \right) = \ceil{\frac{n}{M}}\left(a - q - 2 + M \right) \geq n - \ceil{\frac{n}{M}}(q + 2 - a).
\end{align*}
Hence, $p$ is covered at least $(q + 2) - (q + 2 - a) = a$ times by the hyperplane given by $x_1 + \dots x_d = \sum_{i=1}^d a_i$, and so is covered at least $k$ times.

Suppose now that $d$ is even. Here, we only give the $k$-cover. The verification proceeds similarly to the case when $d$ is odd.

Assume $k= (d+1)q + r$ with $q, r \in \ZZ$ and $1\leq r \leq d/2 + 1$. Set $M := (d+1)(q+1) - (r-1)$. For each $j \in \{1, \dots, 2q+1\}$, let $B_j$ consist of the following:
\begin{itemize}
\item hyperplanes given by $x_i = c$, for $1 \leq i \leq d$ and $c\in\{(j-1)\ceil{\frac{n}{M}},\dots,j\ceil{\frac{n}{M}} - 1\}$,
\item hyperplanes given by $x_1+\dots+x_d = n -1 - j$, for $c\in\{(j-1)\ceil{\frac{n}{M}},\dots,j\ceil{\frac{n}{M}} - 1\}$.
\end{itemize}
The $k$-cover for this case contains, for each $j$, the hyperplanes in $B_j$ with multiplicity $2q+2-j$. Every other hyperplane has multiplicity zero.

Finally, assume $k= (d+1)q + r$ with $q, r \in \ZZ$ and $d/2 + 1 < r \leq d+1$. Set $M := (d+1)(q+1)+(d+2-r)$. For each $j \in \{1, \dots, 2q+2\}$, let $B_j$ consist of the following:
\begin{itemize}
\item hyperplanes given by $x_i = c$, for $1 \leq i \leq d$ and $c\in\{(j-1)\ceil{\frac{n}{M}},\dots,j\ceil{\frac{n}{M}} - 1\}$,
\item hyperplanes given by $x_1+\dots+x_d = n -1 - j$, for $c\in\{(j-1)\ceil{\frac{n}{M}},\dots,j\ceil{\frac{n}{M}} - 1\}$.
\end{itemize}
The $k$-cover $\mathcal{C}$ contains, for each $j$, the hyperplanes in $B_j$ with multiplicity $2q + 3 - j$. Every other hyperplane has multiplicity zero. \end{proof}

Finally, when $d$ is sufficiently large compared to $k$, we obtain an asymptotic formula for $f(n,d,k)$.

\begin{theorem}\hfill\label{3below}
\begin{enumerate}[label = (\alph*)]
\item \label{3below:a} If $k \geq 2$ and $d \geq 2k-3$, then $$f(n,d,k) = \left(1+\frac{k-1}{d-k+2}\right)n+O_{d,k}(1),$$ 
\item \label{3below:b} If $k \geq 3$ and $2k-3\ge d\ge k-2$, then
\[
f(n,d,k)=\left(2+\frac{2k-3-d}{2d+3-k}\right)n+O_{d,k}(1).
\]
\end{enumerate}
\end{theorem}

\begin{proof}[Proof of Theorem \ref{3below}] We begin with the proof of the lower bound for \ref{3below:a}. By Theorem~\ref{prop:1ce2ce}, the statement holds for $k\in\{2,3\}$. Now assume $k>3$ and proceed by induction.

We first prove that $f(n,2k-3,k)\geq{2n+O_k(1)}$. In order to establish a lower bound of the form ${2n+O_k(1)}$, we may assume, by Observation~\ref{boundary}, that each bounding hyperplane has multiplicity at most one. If some bounding hyperplane has multiplicity zero in a $k$-cover $\mathcal{C}$ of $T_{2k-3}(n)$, Observation~\ref{d-2} gives that $|\mathcal{C}|\geq{f(n,2k-4,k)}$.

We first establish that $f(n,2k-4,k)\ge 2n+O_k(1)$, allowing us to assume that each bounding hyperplane has multiplicity exactly one.
By the induction hypothesis, we have that $f(n,2k-5,k-1)=2n+O_k(1)$, so $f(n, 2k-5, k)\geq{2n}+O_{k}(1)$. To establish a lower bound of the form $2n+O_k(1)$ for $f(n,2k-4,k)$, we may assume, by Observation \ref{boundary}, that each bounding hyperplane is used at most once. 
Let $\mathcal{C'}$ be a $k$-cover of $T_{2k-4}(n)$. If some bounding hyperplane has multiplicity zero in $\mathcal{C'}$, Observation \ref{d-2} gives that $|\mathcal{C'}|\geq{f(n,2k-5,k)}\geq{2n}+O_k(1)$. Thus, we may assume that every bounding hyperplane has multiplicity exactly one in $\mathcal{C'}$. Observe that for any bounding hyperplane $H$, $H\cap T_{2k-4}(n)$ is a copy of $T_{2k-5}(n)$, the interior points of which constitute a copy of $T_{2k-5}(n-(2k-4))$ which still needs to be covered $k-1$ times. By Observation \ref{d-2}, this requires at least $f(n-(2k-4),2k-5,k-1)$ hyperplanes and any non-bounding hyperplane in $\mathcal{C'}$ can intersect at most $2k-4$ of the faces of $T_{2k-4}(n)$, so we obtain
 \begin{align*}
    |\mathcal{C}'| & \geq \left(\frac{2k-3}{2k-4}\right)f(n - (2k-4), 2k-5, k-1)\\ 
    & \geq \left(\frac{2k-3}{2k-4}\right) (2 (n - (2k - 4)) + O_k(1)) \\
    & \geq \left(1 + \frac{1}{2k-4}\right)2n + O_k(1).    
    \end{align*}
    Hence we obtain that $f(n,2k-4,k) \geq {2n+O_k(1)}$.
    
From here on we may now assume that each bounding hyperplane in $\mathcal{C}$ has multiplicity exactly one. The intersection of a bounding hyperplane with $T_{2k-3}(n)$ is a copy of $T_{2k-4}(n)$, the interior points of which constitute a copy of $T_{2k-4}(n-(2k-3))$ which still needs to be covered $k-1$ times. This requires at least $f(n-(2k-3), 2k-4, k-1)$ hyperplanes and any non-bounding hyperplane in $\mathcal{C}$ can intersect at most $2k-3$ of the faces of $T_{2k-3}(n)$, so we obtain 
\begin{align*}
    f(n, 2k-3, k) & \geq \left(\frac{(2k-3)+1}{2k-3}\right)f(n - (2k-3), 2k-4, k-1)\\ 
    & \geq \left(\frac{2k-2}{2k-3}\right) \left(\left(1 + \frac{k-2}{k-1}\right)n + O_k(1)\right) \\
    & = 2n + O_k(1).    
\end{align*}
Suppose now that $d > 2k-3$ and proceed by induction on $d$. In order to establish the lower bound, by Observation \ref{boundary}, we may assume that each bounding hyperplane has multiplicity at most one. Furthermore, if the hyperplane $H$ given by $x_1=0$ is not contained in a $k$-cover $\mathcal{C}$, we have, by Observation \ref{d-2} that $$|\mathcal{C}|\geq{f(n,d-1,k)}= \left(1+\frac{k-1}{d-k+1}\right)n+O_{d,k}(1) \geq {\left(1+\frac{k-1}{d-k+2}\right)n+O_{d,k}(1)} .$$
Thus, we can assume that the multiplicity of $H$ in $\mathcal{C}$ is exactly one. Let $s$ be the smallest nonnegative integer such that $x_1 = s$ has multiplicity zero. Such an $s$ exists and, furthermore, $s \leq n-1$ 
    since, without loss of generality, the hyperplane given by $x_1 = n-1$ has multiplicity zero.
Since the points in $T_d(n)\cap H$ still need to be covered $k-1$ times,
    \begin{equation}\label{eq:highdimkcover1}
        |\mathcal{C}| \geq s + f(n, d-1, k-1) \geq s + \left(1 + \frac{k-2}{d-k+2}\right)n + O_{d,k}(1).
    \end{equation} 
    On the other hand, if $H_s'$ is the hyperplane given by $x_1 = s$, then the points in $T_d(n) \cap H_s'$ need to be covered at least $k$ times, implying
    \begin{equation}\label{eq:highdimkcover2}
        |\mathcal{C}| \geq s + f(n - s, d-1, k) \geq s + \left(1 + \frac{k-1}{d-k+1}\right)(n-s) + O_{d,k}(1).
    \end{equation} 

    Multiplying (\ref{eq:highdimkcover1}) by $\frac{k-1}{d-k+1}$ and adding it to (\ref{eq:highdimkcover2}) eliminates $s$, and the inequality simplifies to
    $$|\mathcal{C}| \geq \left(1 + \frac{k-1}{d-k+2}\right)n + O_{d,k}(1). $$
This completes the proof of the lower bound for \ref{3below:a}.

For the upper bound for \ref{3below:a}, we rely on Proposition \ref{upboundall}. If $d=2k-3$, then $d$ is odd and we have $k= (d+1)q/2 + r$ with $q= r =1$. If $d\geq{2k-2}$ and $d$ is odd, then $k= (d+1)q/2 + r$ with $q = 0, r = k$. If $d\geq{2k-2}$ and $d$ is even, then $k = q(d+1) + r$ with $q = 0$ and $r = k \leq d/2 + 1$. In all cases, Proposition \ref{upboundall} implies the desired upper bound.

     \emph{Proof of~\ref{3below:b}:} For the lower bound, we proceed via induction on $k$ where the base case of $k=3$ has already been established --- $d=1$ is straightforward, $d=2$ is given by Theorem \ref{2int}, and $d=3$ comes from Theorem \ref{prop:1ce2ce}\ref{itm:highdim3}.
For a given $k$, we first assume $k-2\le d\le 2k-6$. By Observation \ref{boundary}, it suffices to assume each bounding hyperplane has multiplicity at most two in the cover $\mathcal{C}$. For any bounding hyperplane $H$, $H\cap T_d(n)$ is a copy of $T_{d-1}(n)$, the interior points of which constitute a copy of $T_{d-1}(n-d)$, which still needs to be covered at least $k-2$ additional times. By Observation~\ref{d-2}, this requires at least $f(n-d,d-1,k-2)$ hyperplanes and any non-bounding hyperplane in $\mathcal{C}$ can intersect at most $d$ faces of $T_d(n)$, so we obtain
\[
|\mathcal{C}|\ge \left(\frac{d+1}{d}\right)f(n-d,d-1,k-2).
\]

If $k=4$, then $d=2$, and we already have the desired result from Theorem \ref{2int}. Otherwise, we note that $k-2\ge{3}$ and $2(k-2)-3\ge d-1 \ge (k-2)-2$. The induction hypothesis gives:
\begin{align*}
    |\mathcal{C}|& \ge \left(\frac{d+1}{d}\right)\left(\left(2+\frac{2(k-2)-3-(d-1)}{2(d-1)+3-(k-2)}\right)(n-d)+O_{d,k}(1)\right)\\
    &=\left(2+\frac{2k-3-d}{2d+3-k}\right)n+O_{d,k}(1),
\end{align*}
as desired.

For $d=2k-5$, we let $s$ (resp. $t$) be the smallest non-negative integer such that the hyperplane $x_1=s$ (resp. $x_1=t$) has multiplicity at most one (resp. has multiplicity zero) in $\mathcal{C}$. This yields
\[|\mathcal{C}|\ge s+t +f (n-s,d-1,k-1)\]
and \[|\mathcal{C}|\ge s+t +f(n-t,d-1,k).\] The asymptotics for $f(n-s,d-1,k-1)$ and $f(n-t,d-1,k)$ are given by the induction hypothesis. We may then assume that $|\mathcal{C}|<\left(2+\frac{2k-3-d}{2d+3-k}\right)n+O_k(1)$ and proceed as in the proof of the $k=3$ case of Theorem \ref{2int} to derive a contradiction. Having now established that the induction hypothesis holds for $d=2k-5$, we repeat the exact same process for $d=2k-4$. We omit the details of these calculations.

For the upper bound, we again rely on Proposition \ref{upboundall}. Suppose $d = k-2$. If $d$ is odd, we have $k= (d+1)q/2 + r$ with $q = 2, r=1$. If $d$ is even, we have $k=(d+1)q+r$ with $q=1, r=1$. Suppose now that $k-2 < d \leq 2k-3$. If $d$ is odd, then $k= (d+1)q/2 + r$ with $q = 1, r = k - (d + 1/2)$. If $d$ is even, then $k = q(d+1) + r$ with $q = 0$ and $r = k > d/2 + 1$. In all cases, Proposition \ref{upboundall} implies the desired upper bound.

\end{proof}

Finally, we remark that for fixed $n$ and $k$, the function $f(n,d,k)$ is non-increasing in $d$ --- any collection of hyperplane equations constituting a $k$-cover in dimension $d$ also gives a $k$-cover in dimension $d+1$. Along with Theorem \ref{3below}\ref{3below:b}, this establishes that $f(n,d,k)\ge 3n+O_k(1)$ for $d<k-2$.

\section{Conclusion and open questions} \label{sec:oq}

In Proposition~\ref{upper}, we obtained the upper bound $f(n,2,k)\le f^*(k,2)n+O_k(1)$. In Section~\ref{sec:highdim}, we show that, for $d \geq 3$, it is not true that $f(n,d,k)\leq{f^*(k,d)n+O_{d,k}(1)}$. Perhaps instead, it is the case that the solution to the problem of fractionally covering $T_d(k)$ gives a lower bound for the $k$-covering problem on $T_d(n)$.

\begin{question}
    \label{q:cdkfractionalnew}
    Is $f(n,d,k)\ge f^*(k,d)n$ for all $n,d,k$?
\end{question}

A positive answer to Question~\ref{q:cdkfractionalnew} for $d=2$ would establish Conjecture \ref{duality}.

We have shown asymptotics for $f(n,3,k)$ for $k\le 5$ and similar techniques yield $f(n,3,7)=4n+O(1)$. On the basis of these results and computations performed with Gurobi~\cite{G}, we pose the following conjecture.

\begin{conjecture} \label{c:d=3}
For all $n, k \geq 1$,
\[f(n,3,k)=
\begin{dcases}
\left(\frac{k+1}{2}\right)n+O_k(1)&\mbox{if } k \mbox{ is odd,}\\
\left(\frac{k(k+2)}{2(k+1)}\right)n+O_k(1)&\mbox{if } k \mbox{ is even.}
\end{dcases}\]
\end{conjecture}

If Conjecture \ref{c:d=3} holds, we note that $f(n,3,k)$ is linear in $n$ up to a constant term and that the slope is integral when $k\equiv 1 \pmod{2}$. Recall that we have established that $f(n,5,1)=n$, $f(n,5,4)=2n+O(1)$, and $f(n,5,7)=3n+O(1)$. Similarly, we also have $f(n,7,1)=n, f(n,7,5)=2n+O(1)$, and $f(n,7,9)=3n+O(1)$. This leads to the next question.
\begin{question}\label{q:1modhalf}
    Is $f(n,5,3m+1)=(m+1)n+O_m(1)$ for all $m$? More generally, for all $m$ and $b$, is $f(n,2b-1,bm+1)=(m+1)n+O_m(1)$?
\end{question}

In Proposition \ref{upboundall}, we established a linear upper bound on $f(n,d,k)$, and we have the trivial linear lower bound $f(n,d,k)\geq f(n,d)=n$. We conjecture that $f(n,d,k)$ is always linear in $n$.
\begin{conjecture}\label{c:constantexists}
    For all $d,k$, there exists a constant $C_{d,k}$ such that
    \[
    f(n,d,k)=C_{d,k}n+O_{d,k}(1).
    \]
\end{conjecture}

In the cases where we know such $C_{d,k}$ exists, we notice a few patterns and ask if they continue. \begin{question}\label{q:conditional1}
    Is $C_{d,k}=\frac{2k}{d+1}+O_d(1)$?
\end{question}
\begin{question}\label{q:conditional2}
    For all $k$, is $C_{d,k+(d+1)}\ge C_{d,k}+2$?
\end{question}

Lastly, whenever we have provided a formula for $f(n,d,k)$, the constructions given in the proof of Proposition~\ref{upboundall} are optimal, up to $O_{d,k}(1)$ terms. If this is the case for all $d, k$, then Conjectures~\ref{c:d=3}~and~\ref{c:constantexists} are true and Questions \ref{q:1modhalf}, \ref{q:conditional1}, and \ref{q:conditional2} can be answered in the affirmative. We conjecture that this is indeed the case, thus conjecturing an asymptotic formula for $f(n,d,k)$ for all pairs of $d$ and $k$.

\begin{conjecture}\label{megaconjecture}
Suppose $d$ is odd and $k= (d+1)q/2 + r$ with $q, r \in \ZZ$ and $1\leq r \leq (d+1)/2$. Then $f(n,d,k)= {C_{d,k}n}+O_{d,k}(1)$ where 
   \[ C_{d,k}=(q+1)\left(1+\frac{r-1}{(d+1)(q+2)/2 - (r-1)}\right). \]
   Suppose $d$ is even and $k=(d+1)q+r$ with $q, r \in \ZZ$ and $1\leq r\leq {d+1}$. Then $f(n,d,k)= {C_{d,k}n}+O_{d,k}(1)$ where
$$ C_{d,k}=\begin{dcases} (2q+1)\left(1 +\frac{r-1}{(d+1)(q+1) - (r-1)} \right)&\mbox{if }1\leq{r}\leq{d/2+1}\mbox{,}\\
(2q+3)\left(1-\frac{d+2-r}{(d+1)(q+1)+(d+2-r)}\right)&\mbox{otherwise.}
\end{dcases}$$
\end{conjecture}

\section*{Acknowledgments}
This project began at the 2021 Graduate Research Workshop in Combinatorics. Some of the results in Section \ref{twodim} were included in the second author's PhD dissertation \cite{C}.

\end{document}